\documentclass[10pt,reqno,twoside]{amsart}
\newcommand{\Lo}{{ \log(\mbox{$\frac{T}{2 \pi}$})}}

\newtheorem{theorem}{Theorem}
 
 \newtheorem{lemma}[theorem]{Lemma}
 
 \newtheorem{proposition}[theorem]{Proposition}
  \theoremstyle{remark}

\reversemarginpar 

\begin{document}

\title{Extreme values of $\zeta'(\rho)$}
\author{Nathan Ng}
\maketitle
{\def\thefootnote{}
\footnote{\today. \ {\it Mathematics Subject Classification (2000)}.
11M26}}
\begin{abstract}
\noindent
In this article we exhibit small and large values of $\zeta'(\rho)$ by
applying Soundararajan's resonance method.  Our results assume the 
Riemann hypothesis.  
\end{abstract}
\section{Introduction}

Let $\zeta(s)$ denote the Riemann zeta function and let $\rho$
denote a non-trivial zero of this function.   A famous conjecture
due to Riemann asserts that all non-trivial zeros $\rho$ have
real part equal to one-half.  This is the Riemann hypothesis. 
In this article we are concerned with large and small values of $\zeta'(\rho)$.  Note that if $|\zeta'(\rho)|$ were small then 
we would expect a small gap between consecutive zeros of $\zeta(s)$ nearby.
An extreme example of this phenomenon is 
that if $\rho$ is a multiple zero of the zeta function
then $\zeta'(\rho)=0$. On the other hand, 
if $\zeta'(\rho)$ were large we would expect a large gap
between zeros of $\zeta(s)$ nearby.   This has been observed numerically
in Odlyzko \cite{Od}.  Also Soundararajan \cite{So0} has conjectured that
a zero of $\zeta'(s)$ close to the half line would correspond to nearby pair
of close zeros of the zeta function on the half-line.  Recall that 
the phenomenon of a close pair of zeros 
of $\zeta(s)$ is referred to as Lehmer's phenomenon.  
One reason for our interest in such
small spaces between the zeros of zeta is due to their connection
to the non-existence of  Landau-Siegel zeros.  This connection was first
noticed by Montgomery in \cite{Mo} and 
Montgomery and Weinberger in \cite{MW}.
This idea was further
explored by Conrey and Iwaniec in \cite{CI}.
The problem of the true size of $\zeta'(\rho)$ remains an open question.  Under
the Riemann hypothesis, we have by an argument of Littlewood, that there exists $c_0 > 0$ such that 
\[   |\zeta'(\rho)| \ll \exp \left( \frac{c_0 \log |\gamma|}{\log \log |\gamma|}
   \right) 
\]
where $\gamma = \mathrm{Im}(\rho)$.  This last notation shall be employed
throughout the article.  On the other hand, we are also interested in small values of $|\zeta'(\rho)|$.  Consider $
   \Theta = \inf \{  \  c  \ | \   |\zeta'(\rho)|^{-1} \ll \gamma^{c} \ \} 
$ defined by Gonek \cite{Go} in his study of $M(x)$, the summatory function of 
the M\"{o}bius function.   Since the Riemann hypothesis implies $|\zeta'(\rho)|
\ll |\rho|^{\epsilon}$ one expects that $\Theta \ge 0$. On the other hand,
the GUE conjecture which asserts that the that distribution of consecutive zeros of the zeta function obey the GUE distribution suggests that $\Theta =\frac{1}{3}$ and hence we should have 
\begin{equation}
     |\zeta'(\rho)| \ll \gamma^{-1/3+\epsilon}
 \nonumber
\end{equation}
infinitely often. 

In this article we shall produce results exhibiting both large and small values 
of $|\zeta'(\rho)|$.  These results are obtained by a novel idea due to Soundararajan \cite{S}. The method, coined the resonance method, will be explained shortly.   
We begin with the large values result. 

\begin{theorem} \label{large1}  Assume the Riemann hypothesis. 
For each $A > 0$ we have 
\[
   |\zeta'(\rho)| \gg_{A}  (\log |\gamma|)^{A}
\]
for infinitely many $\gamma$.  
\end{theorem}
I would like to note that Soundararajan has informed me that 
he has proven that 
\begin{equation}
    \sum_{0 < \gamma < T} |\zeta'(\rho)|^{2k} \gg_{k}
    T (\log T)^{(k+1)^2}  
    \label{eq:lbm}
\end{equation}
by the lower bound method of Rudnick and Soundararajan \cite{RS1}, \cite{RS2}.
Clearly, (\ref{eq:lbm}) implies Theorem 1. However, as this remains unpublished,
we present our proof of Theorem \ref{large1}.  
Thus under the Riemann hypothesis, the lower bound method \cite{RS1},
\cite{RS2} can give omega results for $\zeta'(\rho)$ of the same strength
as the resonance method \cite{S}.  This stems from the fact that we are 
unable to evaluate a certain weighted sum of $\zeta'(\rho)$ without
making assumptions about the zeros of Dirichlet $L$-functions
(see Proposition \ref{prop1} parts $(ii)$ and $(iii)$ that follow).  
If we are willing to assume an additional hypothesis 
concerning the location of the zeros of Dirichlet $L$-functions 
we can improve Theorem  \ref{large1} significantly and we can 
obtain a result of the same quality as Soundararajan's
results \cite{S}.  We shall require the following: 

\noindent {\bf Large zero-free region conjecture}. 
There exists a positive constant $c_0'$ sufficiently large such that
for each $q \ge 1$ and each character $\chi$ modulo $q$ the Dirichlet $L$-function
$L(s,\chi)$ does not vanish in the region 
\[
     \sigma \ge 1 - \frac{c_0'}{\log \log (q (|t|+4))}
\]
where $s= \sigma + it$. 

Note that this conjecture is significantly weaker than the generalized 
Riemann hypothesis. However, it is a sufficiently strong hypothesis 
to rule out the existence of Siegel  zeros.  
Recall that the classical zero-free region for Dirichlet 
$L$-functions is  $L(s,\chi)$ does not vanish in the region
\[
   \sigma \ge 1 - \frac{c_1}{\log(q (|t|+3))}
\]
for some $c_1 >0$ 
with the possible exception of one simple real zero in the case 
$\chi$ is quadratic.  

\begin{theorem} \label{large2}
Assume the Riemann hypothesis and the large zero-free region conjecture.  
There are arbitrarily large values of $\gamma$
such that 
\[
     |\zeta'(\rho)| \gg 
     \exp \left(c_2  \sqrt{ \frac{\log |\gamma|}{\log \log |\gamma|}
     } \right) 
\]     
where $c_2 = \frac{1}{\sqrt{2}} - \epsilon$ is valid. 
\end{theorem}

We also prove a result for small values of $|\zeta'(\rho)|$. Surprisingly,
this proof is significantly easier than the proof
of Theorem \ref{large2}. 
\begin{theorem} \label{small}  Assume the Riemann hypothesis.  
We have  
\[
    |\zeta'(\rho)| \ll 
    \exp \left(-c_3 \sqrt{ \frac{\log |\gamma|}{\log \log |\gamma|}
     } \right) 
\]    
for infinitely many $\gamma$ where $c_3 = \sqrt{\frac{2}{3}}-\epsilon$ is valid.  
\end{theorem}
\section{Notation}
We shall use Vinogradov's notation $f(x) \ll g(x)$ to mean there exists a $C>0$
such that
$
    |f(x)| \le C g(x)
$
for all $x$ sufficiently large.  We denote $f(x)=O(g(x))$ to mean the same thing. 
Also, $f(x)=o(g(x))$ means $f(x)/g(x) \to 0$ as $x \to \infty$. We shall consider arbitrary sequences $x_n$  supported on an interval $[1,M]$ and we employ 
the notation 
\[
    ||x_n ||_{\infty} = \max_{n \le M} |x_n| \ \mathrm{and} \
    ||x_n||_{p} = ( \sum_{n \le M} |x_n|^{p} )^{1/p} \ . 
\]
 
We now define some basic arithmetic functions.  We define $\mu(n)$, 
the Mobius function, to be the coefficient of $n^{-s}$ in the Dirichlet 
series $\zeta(s)^{-1} = \sum_{n=1}^{\infty} \mu(n) n^{-s}$. 
We define $\Lambda_{k}(n)$ to be the coefficient of $n^{-s}$ in the Dirichlet series of $(-1)^{k}\zeta^{(k)}(s)/\zeta(s)$.  Another way to express this is $\Lambda_k(n) = (\mu*\log^k)(n)$.
Note that $\Lambda_{k}(n)$ is supported on those integers with at most 
$k$ distinct prime factors.   We define $\tau_k(n)$, the $k$-the divisor 
function, to be the coefficient of $n^{-s}$ in the Dirichlet series 
$\zeta(s)^k = \sum_{n=1}^{\infty} \tau_k(n) n^{-s}$.

\section{Explanation of the resonance method}
In this section we outline the resonance method.  Soundararajan
\cite{S} recently invented this simple method to find large 
values of $|\zeta(1/2+it)|$
(and also other $L$-functions and character sums).  
Under the Riemann hypothesis it is known that 
\[
    |\zeta(1/2+it)| \ll \exp \left(\frac{c_1' \log |t|}{\log \log |t|}
    \right)  
\]
where $c_1' > 0$ is explicitly given.
However, it has been proven by Montgomery \cite{Mo2}, assuming the Riemann 
hypothesis, that there exist 
arbitrarily large $t$ such that
\[
    |\zeta(1/2+it)| \gg  
    \exp
    \left(  c_2'
    \sqrt{ \frac{\log |t|}{\log \log |t|} }
    \right)
\]   
for some positive constant $c_2'$.  
Later, Balasubramanian and  Ramachandra \cite{BR} gave an unconditional
proof of this result with an explicit value $c_{2}' <1$.  
The new method permits the choice 
$c_2' = 1-\epsilon$.  We now sketch the method. Consider the mean values
\[
    \int_{T}^{2T} \zeta(1/2+it) |A(it)|^2 \, dt \ \mathrm{and}
    \ \int_{T}^{2T} |A(it)|^2 \, dt 
\]
where $A(s) = \sum_{n \le M} x_n n^{-s}$ is a Dirichlet polynomial
with arbitrary positive coefficients $x_n$ and $y \le T^{1-\epsilon}$.  A standard
calculation shows that 
\[
   \frac{   \int_{T}^{2T} \zeta(1/2+it) |A(it)|^2 \, dt  }{\int_{T}^{2T} |A(it)|^2 \, dt
  }  =  \left(
  \frac{ \sum_{nu \le M} \frac{x_n x_{nu}}{\sqrt{u}}}{\sum_{n \le M} x_{n}^2}
  \right)
  (1+o(1))  \ . 
\] 
By taking absolute values we deduce that
\begin{equation}
    \max_{T \le t \le 2T} |\zeta(1/2+it)| 
    \ge  \left(
    \frac{ \sum_{nu \le M} \frac{x_n x_{nu}}{\sqrt{u}}}{\sum_{n \le M} x_{n}^2}
    \right)
  (1+o(1)) \ . 
   \label{eq:quadf}
\end{equation}
The problem is thus reduced to optimizing the fraction on the right. 
Soundararajan \cite{S} shows that the maximum of the above quotient is 
\[
   \exp \left( \mbox{$
  \sqrt{ \frac{\log M}{\log \log M}} $} (1+o(1))
   \right)
\]
and this is obtained by choosing $x_n=f(n)$ where 
$f(n)$ is multiplicative and supported
on squarefree numbers.  We define $f$ on the primes as follows:
let $L = \sqrt{\log M \log \log M}$
and set
\[
   f(p) =  \left\{ \begin{array}{ll}
                  \frac{L}{\sqrt{p} \log p} & 
                  \mbox{if $L^2 \le p \le \exp((\log L)^2)$} \\
                  0 & \mbox{else} \\
                  \end{array} \right. \ . 
\]

The strategy of this article is to follow the above argument. We require
asymptotic formulae for the mean values 
\[
  S_1 = \sum_{0 < \gamma < T} \zeta'(\rho)A(\rho)A(1-\rho)
  \  \mathrm{and} \ 
  S_2 = \sum_{0 < \gamma < T} A(\rho)A(1-\rho)  
\]
where $A(s) = \sum_{n \le M} x_n n^{-s}$ has arbitrary real coefficients $x_n$, $y=T^{\theta}$, and $\theta < 1/2$. 
Observe that if the Riemann Hypothesis is true then
$|A(\rho)|^2= A(\rho)A(1-\rho)$ and thus
\begin{equation}
    S_1 = \sum_{0 < \gamma < T} \zeta'(\rho) |A(\rho)|^{2}
    \ \mathrm{and} \
    S_2 = \sum_{0 < \gamma < T} |A(\rho)|^{2}  \ . 
\end{equation}
In fact, we shall show that $S_1/S_2$ is essentially the same
quotient of quadratic forms as in~(\ref{eq:quadf}).  

We have the following formulae for $S_1$ and $S_2$:
\begin{proposition} \label{prop1}
$(i)$ Suppose that $|x_n| \ll T^{\epsilon}$ and $\theta < 1$.  Then we have 
\begin{equation}
    S_2 = 
    N(T) \sum_{m \le M} \frac{x_{m}^2}{m}
    - \frac{T}{\pi} \sum_{m \le M} \frac{(\Lambda*x)(m)x_{m}} {m}
    + o(T)  
    \label{eq:S2eq}
\end{equation}
where $N(T)$ is the number of zeros of the zeta function in the box
$0 \le \mathrm{Re}(s) \le 1$, $0 \le \mathrm{Im}(s) \le T$.  \\
\noindent $(ii)$ Suppose that $|x_n| \ll \tau_r(n) (\log T)^{C}$ 
for some $C>0$ and 
$\theta < 1/2$.  Then we have 
\begin{equation}
   S_1 = \frac{T}{2 \pi}
  \left( \sum_{nu \le M} \frac{x_{u}x_{nu}r_0(n)}{nu}
   +
   \sum_{{\begin{substack}{a,b \le M
         \\ (a,b)=1}\end{substack}}}
   \frac{r_1(a,b)}{ab}
  \sum_{g \le \min(\frac{M}{a},\frac{M}{v})} \frac{x_{ag}x_{bg}}{g}
  \right)
   + o(T) 
   \label{eq:S1eq}
\end{equation}
where
\begin{eqnarray*}
   r_0(n) & =
   \mbox{$\frac{1}{2}$} P_2( \log(\mbox{$\frac{T}{2 \pi}$}))
   - P_1( \log(\mbox{$\frac{T}{2 \pi}$})) \log n -
   \mbox{$\frac{1}{2}$} (\log n)^2
   +(\Lambda*\log)(n) \ ,  \\
   r_1(a,b) & =   \mbox{$\frac{1}{2}$} \Lambda_2(a)
   - R_1 \left( \mbox{$\log \left( \frac{T}{b} \right)$} \right) \Lambda(a) 
   -\tilde{R}_1 \left( \mbox{$ \log \left( \frac{T}{b} \right) $} \right) \alpha_1(a)
   -\alpha_2(a) \ , 
   \label{eq:r0r1}
\end{eqnarray*}
$P_2,P_1, R_1, \tilde{R}_1$ are monic polynomials of degrees 2,1,1,1 respectively.
$\alpha_2,\alpha_1$ are arithmetic functions.
$\alpha_2$ is supported on $a$ with $\omega(a) \le 2$ and $\alpha_1$ 
is supported on prime powers.  Moreover, $\alpha_1(p^j) \ll \frac{\log p}{p}$,
$\alpha_2(p^j) \ll \frac{j (\log p)^2}{p}$, and $\alpha_2(p^j q^k) \ll (\log p)(\log q)
(p^{-1} + q^{-1})$. \\
(iii) Assume the large zero-free region conjecture. 
The formula for $S_1$ in (ii) remains valid under the assumption that
$x_n =\sqrt{n}f(n) $ and $\theta < 1/3$.
\end{proposition}
\begin{proof}  The proofs of $(ii)$ and $(iii)$ may be found in
Theorem 1.3 of \cite{Ng2}. 
The formula for $S_2$ in $(i)$ 
is mentioned without proof on page 6 of \cite{CGGGH}.
It can be proven by following the argument of \cite{Ng0} Lemma 3. 
\end{proof}
  
From Proposition \ref{prop1}, we can explain our strategy for
proving Theorem \ref{large2}.  
We shall show that in the formulae  (\ref{eq:S1eq}) and (\ref{eq:S2eq})
for $S_1$ and $S_2$ the significant 
terms are
\[
       \frac{T  \log^2(\mbox{$\frac{T}{2 \pi}$})}{4 \pi} 
       \sum_{nu \le M} \frac{x_u x_{nu}}{nu} 
       \ \mathrm{and} \ 
     \frac{T   \log(\mbox{$\frac{T}{2 \pi}$})}{2 \pi}
      \sum_{m \le M} \frac{x_{m}^2}{m}
 \]
respectively.  By choosing $x_n = \sqrt{n}f(n)$ we see that
\[
   \max_{T \le \gamma \le 2T} |\zeta'(\rho)|
   \ge \frac{S_1}{S_2} 
   \approx  
 \frac{\log(\frac{T}{2\pi})}{2}  \left(
  \frac{ \sum_{rn \le M} \frac{f(n)f(nr)}{\sqrt{r}}}{\sum_{n \le M} f(n)^2}
  \right)
  = \exp \left( \mbox{$\sqrt{\frac{\log M}{\log \log M}}$}  (1+o(1))
  \right) \ . 
\]
This is the essential content of Theorem \ref{large2}.  
In order to make this argument rigorous, we will show that each 
of the other terms in the formulae for $S_1$ and $S_2$ are smaller than the principal terms. 
The argument for Theorem \ref{small} is very similar.  In this case we consider 
 \[
     S_3 = \sum_{T < \gamma < 2T} \zeta'(\rho)^{-1} |A(\rho)|^2 \ 
     \mathrm{and} \
     S_2 = \sum_{T < \gamma < 2T} |A(\rho)|^2 \ . 
 \]
As before we will show that the ratio $S_3/S_2$ gives rise to the same
quadratic form as in (\ref{eq:quadf}).  

\section{Large values of $\zeta'(\rho)$: Proof of Theorem \ref{large1}}

In this section we prove Theorem \ref{large1}.  As explained previously
our strategy is to evaluate asymptotically $S_1/S_2$ for
a certain choice of coefficients.  As we are only assuming the 
Riemann hypothesis, we are restricted to choosing
$x_n = \tau_{r}(n)$ with $r \in \mathbb{N}$.
In the course of this calculation, we shall encounter several other
multiplicative functions.  We define 
\[
     f_1(n) =  \prod_{p^e \mid \mid n} \frac{\sum_{j=0}^{\infty} 
     \frac{\tau_{r}(p^{e+j})}{p^{j}}}
     {
     \sum_{j=0}^{\infty} 
     \frac{\tau_{r}(p^{j})}{p^{j}} 
     }
    \ \mathrm{and} \
   f_2 (n) =  \prod_{p^e \mid \mid n} \frac{\sum_{j=0}^{\infty} 
     \frac{\tau_{r}(p^{e+j}) \tau_{r}(p^{j})}{p^{j}}}
     {
     \sum_{j=0}^{\infty} 
     \frac{\tau_{r}(p^{j})^2}{p^{j}} } \ . 
\]
Note that for $i=1,2$ 
$
   f_i(p)= r(1+O(p^{-1})) 
$.
The asymptotic evaluation of $S_1$  will require the evaluation
of several sums of standard arithmetic functions.
 We shall employ the following:
\begin{lemma} \label{mult} Let $a,b,k,r, u \in \mathbb{N}$. \\
$(i)$ 
\[
    \sum_{n \le x} \tau_{r}(nu) = \frac{f_1(u) x (\log x)^{r-1}}{(r-1)!}     (1 + O( (\log x)^{-1})) \ .  
\]
$(ii)$
\[
    \sum_{n \le x} \tau_{r}(n)f_1(n)
    =  \frac{C_0 x(\log x)^{r^2-1}}{(r^2-1)!} (1+ O((\log x)^{-1})) 
\]
where 
\begin{equation}
     C_0 = \prod_{p} \left( 1- \frac{1}{p}  \right)^{r}
     \sum_{j=0}^{\infty} \frac{\tau_{r}(p^j) f_1(p^j)}{p^j}
     = \prod_{p} \left( 1- \frac{1}{p}  \right)^{r^2+r}
     \sum_{j=0}^{\infty} \frac{\tau_{r}(p^j) \tau_{r+1}(p^j)}{p^j}
     \ . 
     \label{eq:C0}
\end{equation}
$(iii)$
\[
    \sum_{n \le x} f_2 (n) = \frac{C_1 x (\log x)^{r-1}}{(r-1)!} 
    (1 + O((\log x)^{-1})   
\]
where 
\begin{equation}
    C_1 = \prod_{p}  \left( 1- 1/p\right)^{r} 
    \sum_{j=0}^{\infty} \frac{f_2 (p^j)}{p^j} 
    = \prod_{p} \left(1- 1/p \right)^r 
    \frac{\sum_{j=0}^{\infty} 
    \frac{\tau_{r}(p^j) \tau_{r+1}(p^j)}{p^j} }{
    \sum_{j=0}^{\infty} \frac{\tau_r (p^j)^2}{p^j}
    } \ . 
    \label{eq:C1}
\end{equation}
$(iv)$ 
\[
    \sum_{n \le x} \frac{\tau_{r}(an) \tau_{r}(bn)}{n}
    = C_2 f_2(a) f_2(b)  \frac{(\log x)^{r^2}}{(r^2)!}
    (1 + O((\log x)^{-1})) 
\]
where 
\[
   C_2 = \prod_{p} \left(
   1-1/p
   \right)^{r^2} \sum_{j=0}^{\infty} \frac{\tau_{r}(p^k)^2}{p^k} \ . 
\]
Notice that it follows immediately from
(\ref{eq:C0}) and (\ref{eq:C1}) that $C_0 = C_1 C_2$.   \\
$(v)$
\[
    \sum_{n \le x} \Lambda_{k}(n)
    = k x (\log x)^{k-1} (1 + O((\log x)^{-1})) \ . 
\]
$(vi)$ For $i=1,2$
\[
    \sum_{n \le x} \Lambda(n) f_i (n) 
    = rx (1 + O((\log x)^{-1})) \ . 
\]
$(vii)$ For $i=1,2$
\[
    \sum_{n \le x} \Lambda_2(n) f_i(n)
    = (r^2 + r) x (\log x)(1 + O((\log x)^{-1}))  \ . 
\]
\end{lemma}
\begin{proof}  Since the proofs of $(i)-(iv)$ are very similar we shall
just prove part $(iv)$.   We give a sketch of the proof as the argument is
standard (see for example \cite{Se}).  We define the Dirichlet series $H(s) = \sum_{n=1}^{\infty} \tau_{r}(an)  \tau_r (bn) n^{-s}$
and since $\tau_r$ is multiplicative we have the factorization 
\[
   H(s) = \prod_{(p,ab)=1} \left( 
   \sum_{k=0}^{\infty} \frac{\tau_r (p^k)^2}{p^{ks}}
   \right)
   \prod_{p^e \mid \mid a}  
   \sum_{k=0}^{\infty} \frac{\tau_r(p^{e+k}) \tau_{r}(p^k)}{p^{ks}}
   \prod_{p^f \mid \mid b}   \sum_{k=0}^{\infty}
   \frac{\tau_r(p^{k}) \tau_{r}(p^{f+k})}{p^{ks}} \ . 
\]
Next we define for $s \in \mathbb{C}$ and $n \in \mathbb{N}$
\[
   F(s;n) = \prod_{p^{e} \mid \mid n}
   \left(
   \frac{ \sum_{k=0}^{\infty} \frac{\tau_r(p^{e+k}) \tau_{r}(p^k)}{p^{ks}}}
   {\sum_{k=0}^{\infty} \frac{\tau_r (p^k)^2}{p^{ks}}   } 
   \right)  \ , \
   G(s) = \prod_{p} (1-1/p^s)^{r^2} 
    \sum_{k=0}^{\infty} \frac{\tau_r (p^k)^2}{p^{ks}} 
\]   
and thus $H(s) = \zeta(s)^{r^2} F(s,ab) G(s)$. Moreover, we notice that $F(1;n) =f_2 (n)$ and $G(1)=C_2$.
By Perron's formula,
\[
    \sum_{n \le x} \frac{\tau_{r}(an) \tau_{r}(bn)}{n}
    = \frac{1}{2 \pi i} \int_{\kappa-iU}^{\kappa+iU}H(s+1) \, 
    \frac{x^s ds}{s} + O \left( 
    \frac{(\log x)^{r^2}}{U} + \frac{1}{x^{1-\epsilon}}
    \left( 1 + x \frac{\log U}{U} \right)
    \right)
\]
with $\kappa=(\log x)^{-1}$.  Let $\Gamma(U)$ denote the contour consisting 
of $s \in \mathbb{C}$ such that $\mathrm{Re}(s)=- \frac{c'}{\log 
(|\mathrm{Im}(s)|+2)}$ 
and $|\mathrm{Im}(s)| \le U$ for an appropriate $c'>0$. 
We deform the contour past $\mathrm{Re}(s)=0$ line to $\Gamma(U)$
picking up the residue at $s=0$.   The residue at $s=0$ equals
\[
    C_2 f_2(a) f_2(b)  \frac{(\log x)^{r^2}}{(r^2)!}
    (1 + O((\log x)^{-1})) 
\]
which corresponds to the main term.  Employing standard bounds for $\zeta(s)$
in the zero-free region we can show that contribution of the integral on $\Gamma(U)$ is smaller than the main term for an appropriate choice of $U$
by at least one factor of $\log x$.  
 Part $(v)$ is a well known fact. Part $(vi)$ follows from the fact that 
$\Lambda$ is supported on the prime powers and $\Lambda(p^j)=\log(p)$.
Part $(vii)$ follows from the fact that
$\Lambda_2$ is supported on those $n$ with $\omega(n) \le 2$ 
and moreover
$\Lambda_2(pq) = 2 \log p \log q$, $\Lambda_2(p) = (\log p)^2$,
and  $f_i(p)= r(1+O(p^{-1}))$ . 
\end{proof}
We are now prepared to prove Theorem 1.  In the course of the proof,
we will encounter the following integrals:
\begin{equation}
\begin{split}
   i(u,v) & := \int_{0}^{1} x^{u}(1-x)^{v} \, dx
   = \frac{u! v!}{(u+v+1)!}  \ , \\
   c_{X}(u,v) &  := \int_{1}^{X} \frac{(\log X/t)^{u} (\log t)^{v}}{t} \, dt 
   = (\log X)^{u+v+1} i(u,v)   
   \label{eq:cX}
\end{split}
\end{equation}
where $X \ge 1$.

\noindent {\it Proof of Theorem \ref{large1}}. 
By Proposition \ref{prop1} we may write $S_1 = \tilde{S}_1 + o(T)$ where
\[
    \tilde{S}_1 = 
   \frac{T}{2 \pi} \left( \sum_{nu \le M} \frac{x_u 
   x_{nu}r_0(n)}{nu}
   + \sum_{{\begin{substack}{a,b \le M
         \\ (a,b)=1}\end{substack}}}
   \frac{r_1(a,b)}{ab}
  \sum_{g \le \min(\frac{M}{a},\frac{M}{v})} \frac{x_{ag}x_{bg}}{g} \right)  
\]
and
\begin{equation}
\begin{split}
      r_0(n) & = \frac{1}{2} P_2(  \log(\mbox{$\frac{T}{2 \pi}$})) + \sum_{d \mid n} g(d) \ , \\
     g(d) & = - (P_1( \log(\mbox{$\frac{T}{2 \pi}$}))+\log d)\Lambda(d) + \frac{\Lambda_2(d)}{2} \ . 
     \label{eq:g}
\end{split}
\end{equation}
Thus we have 
$
   \tilde{S}_1 = \frac{T}{2 \pi} \left( \frac{P_2( \log(T/2 \pi))}{2} T_1 + T_2  + T_3 \right)
$
where
\begin{equation}
\begin{split}
   T_1 & = \sum_{nu \le M} \frac{x_u 
   x_{nu}}{nu} \ ,  \\
   T_2 & =  \sum_{dnu \le M} \frac{g(d)x_{u} x_{dnu}}{dnu} \ ,  \\
   T_3 & = \sum_{{\begin{substack}{a,b \le M
         \\ (a,b)=1}\end{substack}}}
   \frac{r_1(a,b)}{ab}
  \sum_{g \le \min(\frac{M}{a},\frac{M}{v})} \frac{x_{ag}x_{bg}}{g} 
  \ . 
\end{split}
\end{equation}

\subsection{Evaluation of $T_1$}
Now by Lemma \ref{mult} $(i)$ and $(ii)$ we have
\begin{equation}
\begin{split}
  T_1 & 
   = \sum_{u \le M} \frac{\tau_{r}(u)}{u} 
 \int_{1^{-}}^{M/u} t^{-1} d \left(
 \sum_{n \le t} \tau_{r}(nu)\right)
 \\
  &  
  \sim \sum_{u \le M} \frac{\tau_{r}(u) f_1(u)}{u} 
  \frac{ (\log M/u)^{r}}{r!}  
  =  \frac{1}{r!} \int_{1^{-}}^{M} \log(M/t)^r t^{-1} 
  d \left(
  \sum_{u \le t} \tau_{r}(u) f_{1}(u) 
  \right)  \\
  & \sim  \frac{1}{r!}  \int_{1}^{M} \frac{ (\log(M/t)^{r}}{t} 
    \frac{C_0 (\log t)^{r^2-1}}{(r^2-1)!} \, dt  \ . 
    \nonumber
\end{split}
\end{equation}    
By (\ref{eq:cX}) it follows that 
\[    
    T_1 \sim \frac{C_0}{r!(r^2-1)!} c_{M}(r,r^2-1)
    = \frac{C_0 (\log M)^{r^2+r}}{(r^2+r)!} \ . 
\]

\subsection{Evaluation of $T_2$}
Since the calculation of $T_2$ and $T_3$ are rather similar to that 
of $T_1$ we shall not record every step of their calculation. 
By Lemma \ref{mult} $(i)$ we have 
\[
   T_2 
  \sim  \sum_{d \le M} \frac{g(d)}{d} 
   \sum_{u \le M/d} \frac{\tau_r(u)}{u} 
   \frac{ f_1 (du) \log(M/du)^r}{r!}  \ .  
\]
As $g$ is supported on those integers $d$ with $\omega(d) \le 2$ we have 
\begin{equation}
\begin{split}
   T_2 & \sim \sum_{d \le M} \frac{g(d) f_1(d)}{d} 
   \sum_{u \le M/d} \frac{\tau_r(u)}{u} 
   \frac{ f_1 (u) \log(M/du)^r}{r!}   \\
   & = \frac{1}{r!}
   \sum_{d \le M}\frac{g(d)f_1(d)}{d}
   \int_{1}^{M/d} \frac{ \log(M/dt)^{r}}{t}
   \frac{C_0 (\log t)^{r^2-1}}{(r^2-1)!} \, dt  \\
   & = \frac{C_0}{(r^2+r)!} \sum_{d \le M}  \frac{g(d) f_1(d)}{d} 
   (\log M/d)^{r^2 +r} 
   \nonumber
\end{split}
\end{equation}
where we have invoked Lemma \ref{mult} $(ii)$ and (\ref{eq:cX}).  
By (\ref{eq:g}),
Lemma \ref{mult} $(vi)$ and $(vii)$ we obtain
\[
   \sum_{n \le x} g(n) f_1(n)
   \sim   x \left(
   \frac{r^2-r}{2} \log x - r P_1( \log(\mbox{$\frac{T}{2 \pi}$}))
   \right)  \ . 
\]
From this we deduce
\begin{equation}
\begin{split}
   T_2 & =   \frac{C_0}{(r^2+r)!}
   \int_{1}^{M} \frac{(\log M/t)^{r^2+r}}{t}
   \left( 
   \frac{r^2-r}{2} \log(t) - r P_1(\Lo)
   \right) \, dt  \\
   & \sim  \frac{C_0}{(r^2+r)!}
   \left(
   \frac{r^2-r}{2} c_{M}(r^2+r,1)- \frac{r}{\theta} c_{M}(r^2+r,0)
   \right) (\log M)^{r^2+r+2}
   \nonumber
\end{split}
\end{equation}
and it follows from~(\ref{eq:cX}) that 
\[
   T_2 \sim  
   \frac{C_0 (\log M)^{r^2+r+2}}{(r^2+r+2)!}
   \left(
   \frac{r^2-r}{2}-
    \frac{r}{\theta} (r^2+r+2)
   \right) \ . 
\]
\subsection{Evaluation of $T_3$}
By Lemma \ref{mult} $(iv)$ it follows that
\[
     T_3 \sim \frac{C_2}{(r^2)!} \sum_{{\begin{substack}{a,b \le M
         \\ (a,b)=1}\end{substack}}}
   \frac{r_1(a,b) f_2(a) f_2(b)}{ab}   
   \left(\log \min \left(\frac{M}{a},\frac{M}{b} \right) \right)^{r^2}  
\]
where $r_1(a,b)$ is defined by~(\ref{eq:r0r1}).
We shall write this last sum as $T_{3}'+T_{3}''$
where $T_{3}'$ is the sum over the terms for which $a< b \le M$ and 
$T_{3}''$ consists of the terms for which $b < a\le M$.  We have
\[
   T_{3}' \sim  \frac{C_2}{(r^2)!} 
   \sum_{b \le M} \frac{f_2(b)\log(M/b)^{r^2}}{b}
   \sum_{{\begin{substack}{a < b 
         \\ (a,b)=1}\end{substack}}}  \frac{(1/2)\Lambda_2(a)-\Lambda(a)
         R_1(\log(T/b))}{a} f_{1}(a)  
\]
since it may be checked that the contribution from the term
$  -\tilde{R}_1 \left(\log \left(T/b \right) \right) \alpha_1(a)
   -\alpha_2(a)$ is $\ll (\log T)^{r^2+r+1}$. 
By Lemma \ref{mult} $(vi)$ and $(vii)$ 
\[
    \sum_{a \le x} f_{1}(a) 
     \left(
(1/2) \Lambda_2 (a) - \Lambda(a) R_1(\log(T/b))
\right) \sim 
\frac{(r^2+r)}{2} x \log x - r R_1(\log(T/b))x
\]
and it follows that 
\begin{equation}
\begin{split}
   T_{3}' & \sim \frac{C_2}{(r^2)!} 
    \sum_{b \le M} \frac{f_2(b)\log(M/b)^{r^2}}{b}
    \int_{1}^{b} \frac{(1/2)(r^2+r)\log t - r R_1(\log(T/b))}{t} \, dt   \\
    & = \frac{C_2}{(r^2)!} 
    \sum_{b \le M} \frac{f_2(b)\log(M/b)^{r^2}}{b}
    \left(
    \frac{r^2+r}{4} (\log b)^2 - r R_1(\log(T/b)) \log b 
    \right)  \ . 
    \nonumber
\end{split}
\end{equation}
By Lemma \ref{mult} $(iii)$
\begin{equation}
\begin{split}
   T_{3}'  & =  \frac{C_2}{(r^2)!} 
     \int_{1}^{M} \frac{ \log(M/t)^{r^2}}{t}
      \left(
    \frac{r^2+r}{4} (\log t)^2 - r R_1(\log(T/t)) \log t 
    \right) \frac{C_1}{(r-1)!} (\log t)^{r-1} \, dt  \\
   & =  \frac{C_0 (\log M)^{r^2+r+2}}{(r^2+r+2)!}  
    \left(
    \frac{(r^2+5r)(r+1)r}{4}
    -  \frac{r^2  (r^2+r+2)}{\theta}
    \right) \ .
    \nonumber 
\end{split}
\end{equation}
Next, we consider those terms with $b < a \le M$.   We have
\[
   T_{3}'' \sim  \frac{C_2}{(r^2)!} 
   \sum_{a \le M} 
   \sum_{{\begin{substack}{b < a 
         \\ (a,b)=1}\end{substack}}}  \frac{(1/2)\Lambda_2(a)-\Lambda(a)
         R_1(\log(T/b))}{a} \frac{f_1(a)f_2(b)\log(M/a)^{r^2}}{b}
\]
since we can show, as before, that the contribution from the term 
$  -\tilde{R}_1 \left(\log \left(T/b \right) \right) \alpha_1(a)
   -\alpha_2(a)$ is $\ll (\log T)^{r^2+r+1}$.
Since  $\sum_{b \le x} f_2 (b) \sim \frac{C_1}{(r-1)!} x (\log x)^{r-1}$,  a similar calculation as above yields 
\[
     T_{3}'' \sim \frac{C_0}{(r^2)! r!} 
  \int_{1^{-}}^{M} \frac{\log(M/t)^{r^2} (\log t)^r}{t} 
  d \sigma(t)
\]
with 
$
    \sigma(t) = \sum_{a \le t} 
    \left(
    (\Lambda_{2}(a)/2- \Lambda(a) \log T)+ \frac{r}{r+1} \Lambda(a)
    \log(a)
    \right) f_{2}(a) 
$.
By Lemma \ref{mult} $(vi)$ and $(vii)$ 
$
    \sigma(t) \sim 
    \left(
    \frac{r^2+r}{2}+ \frac{r^2}{r+1}  \right)
    t \log t - rt(\log T)  
$
and thus
\begin{equation}
\begin{split}
   T_{3}'' & \sim 
   \frac{C_0}{(r^2)! r!}  \left(
   \left( \frac{r^2+r}{2}+ \frac{r^2}{r+1} \right) c_M(r^2,r+1)
   -r (\log T) c_{M}(r^2,r)  \right) \\
   & = \frac{C_0 (\log M)^{r^2+r+2}}{(r^2+r+2)!}
   \left(
   \left( \frac{(r^2+r)(r+1)}{2}+r^2 \right)
   - \frac{r(r^2+r+2)}{\theta}
  \right)  \ . 
  \nonumber
\end{split}  
\end{equation}
Collecting our results for $T_1,T_2,$ and $T_3=T_{3}'+T_{3}''$ we have 
\begin{equation}
\begin{split}
   S_{1} & \sim \frac{C_0T(\log M)^{r^2+r+2}}{(r^2+r+2)!}
   \left(
   \frac{(r^2+r+2)(r^2+r+1)}{\theta^2}
   + \left(
   \frac{r(r-1)}{2} - \frac{r(r^2+r+2)}{\theta}
   \right)
   \right.  \\
   &+ \left.
   (r^2+r)\left(
   \frac{r^2+5r}{4}+ \frac{r+1}{2}+ \frac{r^2}{r^2+r}
   \right)  - \frac{(r^2+r+2)(r^2+r)}{\theta} \right)  \\
   & \ge   \frac{C_0T(\log M)^{r^2+r+2}}{(r^2+r+2)!}
 \frac{r^2+r+2}{\theta^2}
 (r^2+r+1-\theta(r^2+2r))  \\
 & \gg \frac{r^4 T(\log M)^{r^2+r+2}}{\theta^2 (r^2+r+2)!}
 \nonumber
\end{split}
\end{equation}
for $0 < \theta < \frac{1}{2}$ and $r \in \mathbb{N}$. 
On the other hand, we have the simple bound 
\[
    S_2 \le \frac{T \Lo}{2 \pi} \sum_{m \le M} \frac{\tau_{r}(m)^2}{m}
    \ll \frac{T}{\theta}  (\log M)^{r^2+1} 
\]
and thus
$
   \max_{T \le \gamma \le 2T} |\zeta'(\rho)| 
   \ge \left|   \frac{S_1}{S_2} \right|
   \gg_{r} (\log M)^{r+1} \gg (\log T)^{r+1}  
$.  $\Box$

\section{Larger values of $\zeta'(\rho)$: Proof of Theorem \ref{large2}}

In this section we shall evaluate $S_1/S_2$ for the choice 
$x_n = \sqrt{n} f(n)$.  Before embarking on this task we will 
require a few results concerning the coefficients $f(n)$.  
Moreover, we shall encounter several other multiplicative functions. 
We define $g$ and $h$ to be multiplicative functions supported on the
squarefree numbers. Their values at any prime $p$ are given by
\[
   g(p) = 1+ f(p)^2 \  \mathrm{and} \ 
   h(p) = 1 + f(p)  p^{-1/2}  \ . 
\]
It will also be convenient
to introduce the notation 
\[
   \mathcal{Q}_1 = \prod_{p} 
   \left(
   1+ f(p)^2  + \frac{f(p)}{\sqrt{p}}
   \right) \ , \
    \mathcal{Q}_2 = \prod_{p} \left(1 + f(p)^2 \right) \ . 
\]

\begin{lemma} \label{reso}
$(i)$
\[
    \sum_{nu \le M} \frac{f(u)f(nu)}{\sqrt{n}} 
    = \mathcal{Q}_1 (1+o(1)) \ , 
\]
$(ii)$
\[
    \sum_{n \le M} f(n)^2 \le \mathcal{Q}_2  \ , 
\]
$(iii)$ 
\[
    \frac{\mathcal{Q}_1}{\mathcal{Q}_2} 
    = \exp \left( 
    \mbox{$\sqrt{\frac{\log M}{\log \log M}}$} (1+o(1)) 
    \right) \ .
\]
$(iv)$  For $i=1,2$
\[
       \sum_{a \le M} \frac{\Lambda_i (a) f(a)}{\sqrt{a}g(a)}  \ll 
       (\log T)^{i/2+\epsilon} \ . 
\]
\end{lemma}
\begin{proof}  $(i)$  
We denote the sum to be estimated $\mathcal{S}$. Thus 
\[
   \mathcal{S} = \sum_{n \le M}  \frac{f(n)}{\sqrt{n}} \sum_{{\begin{substack}{u \le M/r 
         \\ (n,u)=1}\end{substack}}} f(u)^2
   =  \sum_{n \le M}  \frac{f(n)}{\sqrt{n}}  
   \left( \prod_{(p,n)=1}(1+f(p)^2)
       - \sum_{{\begin{substack}{u > M/n 
         \\ (n,u)=1}\end{substack}}} f(u)^2 \right) \ . 
\]
By Rankin's trick the error term is bounded by 
\begin{equation} 
  \sum_{n \le M}  \frac{f(n)}{\sqrt{n}}  \left(
    \frac{n}{M} \right)^{\alpha} 
     \sum_{{\begin{substack}{u =1
         \\ (u,n)=1}\end{substack}}} ^{\infty} f(u)^2 u^{\alpha}
   \le \frac{1}{M^{\alpha}} \prod_{p} \left( 
    1 + p^{\alpha} f(p)^2 + f(p) p^{\alpha-1/2} 
    \right)
    \nonumber
\end{equation}
for any $\alpha > 0$.  On the other hand, since $f$ is multiplicative
the main term equals
\[
   \prod_{p}  \left(1+f(p)^2+ \frac{f(p)}{\sqrt{p}} \right)
   + O \left(\frac{1}{M^{\alpha}} 
   \prod_{p}
  \left(1 + f(p)^2 +  \frac{f(p)p^{\alpha}}{\sqrt{p}} \right)  \right)\ .
\]
We deduce 
\begin{equation}
   \mathcal{S} =  \mathcal{Q}_1
   + O \left(\frac{1}{M^{\alpha}} 
   \prod_{p}
    \left(1 + p^{\alpha} f(p)^2 +  \frac{f(p)p^{\alpha}}{\sqrt{p}} \right)  \right) \ .
   \label{eq:Ses}
\end{equation}
However, it is shown in \cite{S} that the ratio of the error term
to the main term in (\ref{eq:Ses}) is $\ll \exp(-\alpha \frac{\log M}{\log \log M})$
for the choice $\alpha = \frac{1}{(\log L)^3}$.  It follows that 
$\mathcal{S} = \mathcal{Q}_1 (1+o(1))$.  \\
$(ii)$  We have the simple identity
\[
   \sum_{n \le M} f(n)^2 \le \sum_{n \ge 1} f(n)^2
   = \mathcal{Q}_2  \ . 
\]
$(iii)$ Note that 
\[
      \frac{\mathcal{Q}_1}{\mathcal{Q}_2} 
      = \prod_{p}
      \left(
      1 +  \frac{f(p)}{\sqrt{p}(1+f(p)^2)}
      \right) \ . 
\]
Taking logarithms of the product  we see that 
\begin{equation}
\begin{split}
    & \log(\mathcal{Q}_1 /\mathcal{Q}_2)
    = \sum_{p} \log \left( 
      1 +  \frac{f(p)}{\sqrt{p}(1+f(p)^2)}
    \right)
    = \sum_{L^2 \le p \le \exp( (\log L)^2)}
    \frac{L}{p \log p(1+o(1))}  \\
    & = \frac{L}{\log L^2} (1+o(1))
    = \sqrt{\frac{\log M}{\log \log M}} (1 + o(1)) \ . 
    \nonumber
\end{split}
\end{equation}
$(iv)$ We have 
\[
   \sum_{a \le M} \frac{\Lambda(a)f(a)}{\sqrt{a}g(a)}
   = L \sum_{p \le M} \frac{1}{p g(p)} \ll L \sum_{p \le M} \frac{1}{p}
   \ll (\log T)^{1/2+\epsilon}  \ . 
\]
Note that $\Lambda_2$ is supported on integers $a$ satisfying 
$\omega (a) \le 2$ and
$f$ is supported on squarefree integers. Moreover  $\Lambda_2 (p) = (\log p)^2$ and  $\Lambda_2 (pq) = 2 
\log p \log q$. From this, we deduce that 
\begin{equation}
\begin{split}
   \sum_{a \le M} \frac{\Lambda_2 (a) f(a)}{\sqrt{a}g(a)}
   &  \ll \sum_{p \le M} \frac{\Lambda_{2}(p)f(p)}{\sqrt{p}g(p)}
    + \sum_{pq \le M, p \ne q} \frac{\Lambda_{2}(pq)f(pq)}{\sqrt{pq}
    g(pq)}  \\
    &  \ll L \sum_{p \le M} \frac{\log p}{p}
    + L^2 \left( \sum_{p \le M} \frac{1}{p} \right)^2 \ll (\log T)^{1+\epsilon}
    \ . 
    \nonumber
\end{split}
\end{equation}
\end{proof}

\noindent {\it Proof of Theorem \ref{large2}.}
We have from Proposition \ref{prop1} that 
\begin{equation}
   \frac{S_1}{S_2} =   \frac{ \frac{1}{2}P_2(\Lo) \Sigma_0
   -  P_1(\Lo) \Sigma_1 - \frac{1}{2}\Sigma_2 + \Sigma_3 + \Sigma_4 
  }{\Lo \Sigma_5-2\Sigma_6 }
   + o(1)
   \nonumber
\end{equation}
where for $i=0,1,2$ 
\[
    \Sigma_i = \sum_{nu \le M} \frac{f(u) f(nu)(\log n)^{i}}{\sqrt{n}} 
\]
and 
\begin{equation}
\begin{split}
    \Sigma_3 & =  \sum_{nu \le M} 
   \frac{f(u)f(nu)(\Lambda*\log)(n)}{\sqrt{n}} \ ,  \\
   \Sigma_4 & =
   \sum_{{\begin{substack}{a,b \le M
         \\ (a,b)=1}\end{substack}}}
   \frac{r_1(a,b)}{\sqrt{ab}}
   \sum_{g \le \min(\frac{M}{a},\frac{M}{b})}  f(ag)f(bg) \ , \\
   \Sigma_{5} & = \sum_{m \le M} f(m)^2  \ , \\
   \Sigma_6 & =\sum_{mn \le M} \frac{\Lambda(n) f(m)f(mn)} {\sqrt{n}}
   \ . 
   \nonumber
\end{split}
\end{equation}
By Lemma \ref{reso} 
\begin{equation}
  \Sigma_{0} = \mathcal{Q}_1 (1 + o(1)) \ \mathrm{and} \ 
  \Sigma_{5} \le \mathcal{Q}_2    (1 + o(1)) \ . 
  \label{eq:sum05}
\end{equation}

We shall prove the following bounds for the other five sums: 
\begin{lemma} \label{sumbds}
We have:
\begin{equation}
\begin{split}
   \Sigma_1  & \ll \mathcal{Q}_1  (\log T)^{1/2+\epsilon}
   \ , \  \Sigma_2, \Sigma_3 \ll \mathcal{Q}_1 (\log T)^{1+\epsilon}
   \ , 
    \\
    \Sigma_4 &  \ll  \mathcal{Q}_1 (\log T)^{3/2+\epsilon} \ , \ 
   \Sigma_6 \ll \mathcal{Q}_2  (\log T)^{1/2+\epsilon}
   \ . 
   \nonumber
\end{split}
\end{equation}
\end{lemma}

\noindent Theorem \ref{large2} now easily follows. 
We deduce from~(\ref{eq:sum05}) and Lemma \ref{sumbds} that 
\[
   \mathcal{S}_1 = (1/2) \mathcal{Q}_1 \log^2(\mbox{$\frac{T}{2 \pi}$})
   \left(
   1+ O((\log T)^{-1/2+\epsilon})
   \right) 
\]
and
$
    \mathcal{S}_2 \le 
    \mathcal{Q}_2 \Lo 
    \left( 1+ O((\log T)^{-1/2+\epsilon})    \right)  
$ .
By Lemma \ref{reso} $(iii)$
\[
   \left| \frac{\mathcal{S}_1}{\mathcal{S}_2} \right|
    \ge  (1/2) \Lo
    \frac{\mathcal{Q}_1}{\mathcal{Q}_2}(1+o(1))
    \ge \exp \left(  \mbox{$
    \sqrt{\frac{\log M}{\log \log M}}$} (1+o(1))
    \right)
\] 
and thus we establish Theorem \ref{large2}. $\Box$  \\

\noindent It suffices to prove Lemma \ref{sumbds}.  \\
\noindent {\it Proof of Lemma \ref{sumbds}}. 
We proceed to bound the various $\Sigma_{i}$.  We begin 
with
\[
    \Sigma_{i}  = \sum_{un \le M} \frac{f(u)f(nu) (\log n)^{i}}{\sqrt{n}} 
\]   
for $i=1,2$.   We evaluate this by writing $(\log n)^{i} 
= \sum_{k \mid n} \Lambda_i(k)$.  Inserting this expression we obtain
\begin{equation}
\begin{split}
  & \Sigma_{i}   = \sum_{k \le M} \frac{\Lambda_i (k) f(k)}{\sqrt{k}}
   \sum_{{\begin{substack}{nu \le M/k
         \\ (nu,k)=1 }\end{substack}}} \frac{f(u) f(nu)}{\sqrt{n}}
   \le \sum_{k \le M} \frac{\Lambda_i(k) f(k)}{\sqrt{k}}
     \sum_{{\begin{substack}{n \le M/k
         \\ (n,k)=1 }\end{substack}}} \frac{f(n)}{\sqrt{n}}
          \sum_{{\begin{substack}{u \ge 1 
         \\ (u,kn)=1 }\end{substack}}} f(u)^2  \\
  & \le  \sum_{k \le M} \frac{\Lambda_i(k) f(k)}{\sqrt{k}}
     \sum_{{\begin{substack}{n \le M/k
         \\ (n,k)=1 }\end{substack}}} \frac{f(n)}{\sqrt{n}}
         \prod_{(p,kn)=1} (1+f(p)^2)
   = \mathcal{Q}_2  \sum_{k \le M} \frac{\Lambda_i(k) f(k)}{\sqrt{k}}
     \sum_{{\begin{substack}{n \le M/k
         \\ (n,k)=1 }\end{substack}}} \frac{f(n)}{\sqrt{n} g(kn)} \\
    & \le  \mathcal{Q}_2  \sum_{k \le M} \frac{\Lambda_i(k) f(k)}{\sqrt{k}g(k)}
     \sum_{n=1}^{\infty} \frac{f(n)}{\sqrt{n} g(n)}
     = \mathcal{Q}_2  
     \prod_{p} \left( 1 + \frac{f(p)}{\sqrt{p} g(p)} \right) 
     \sum_{k \le M} \frac{\Lambda_i(k) f(k)}{\sqrt{k}g(k)} \ . 
     \nonumber
\end{split}
\end{equation}
The expression in front of the last sum is clearly $\mathcal{Q}_1$. Thus 
by Lemma \ref{reso} $(iv)$  
\[
   \Sigma_1 \ll \mathcal{Q}_1 (\log T)^{1/2+\epsilon} \ \mathrm{and} \
   \Sigma_2 \ll \mathcal{Q}_2 (\log T)^{1+\epsilon} \ . 
\]
Next note that  $(\Lambda*\log)(r) \le (\log r)^2$
and hence $\Sigma_3 \le \Sigma_2 \ll (\log T)^{1+\epsilon}$. 
Next we estimate $\Sigma_4$:
\begin{equation}
\begin{split}
   \Sigma_4 & =
    \sum_{{\begin{substack}{a,b \le M
         \\ (a,b)=1}\end{substack}}}
   \frac{r_1(a,b)}{\sqrt{ab}}
    \sum_{g \le \min( \frac{M}{a}, \frac{M}{b} )} f(ag) f(bg)  \\
   & \le \mathcal{Q}_1 
    \sum_{{\begin{substack}{a,b \le M
         \\ (a,b)=1}\end{substack}}}
   \frac{f(a)f(b) |r_1(a,b)|}{\sqrt{ab} g(a)g(b)}
     \\
   & \ll \mathcal{Q}_1  \left(
   \sum_{v \le M} \frac{x_v}{\sqrt{v} g(v)} \right)
   \left(
   \sum_{a \le M} \frac{\Lambda_2 (a) f(a)}{\sqrt{a}g(a)}
   + \log T  \sum_{a \le M} \frac{\Lambda(a) f(a)}{\sqrt{a} g(a)} 
   \right)  \\
   & \le \prod_{p} \left(  1+ \frac{f(p)}{\sqrt{p}} + f(p)^2 \right) 
   (\log T)^{3/2+\epsilon}
    \nonumber
\end{split}
\end{equation}
by Lemma \ref{reso} $(iv)$. 
Finally we have
\[ \Sigma_6 = \sum_{ur \le M} \frac{\Lambda(r) f(u) f(ur)}{\sqrt{r}}
  = \sum_{r \le M} \frac{\Lambda(r)f(r)}{\sqrt{r}} 
    \sum_{{\begin{substack}{u \le  M/r
         \\ (u,r)=1}\end{substack}}} 
      f(u)^2  
      \le  \prod_{p}(1+x_{p}^2)  \sum_{r \le M} 
      \frac{\Lambda(r) f(r)}{\sqrt{r}g(r)} \ .
 \]
Once again by Lemma \ref{reso} $(iv)$ we obtain $\Sigma_6 \ll  \mathcal{Q}_2  
\ll (\log T)^{1/2+\epsilon}$. $\Box$

\section{Small values of $\zeta'(\rho)$: Proof of Theorem \ref{small}.}

\noindent {\it Proof of Theorem \ref{small}.}
We begin by noting that Theorem \ref{small} is automatically true if there
are infinitely many multiple zeros.  Now assume that there are only finitely 
many multiple zeros of $\zeta(s)$.  
Suppose there exists a positive constant $C'$ such that
for all $\gamma > C'$ 
all zeros of the zeta function are simple.  We will now show that for each 
$T$ sufficiently large that there exists a $\gamma \in [T,2T]$ such that
\begin{equation}
     |\zeta'(\rho)|^{-1}  \ge 
     \exp \left(c_5(1+o(1)) 
     \mbox{$\sqrt{ \frac{\log T}{\log \log T}  }$} \right)
     \label{eq:lb}
\end{equation}
and Theorem \ref{small} follows.  
We now establish~(\ref{eq:lb}).  Consider the sums
\[
   S_3 = \sum_{T_1 < \gamma < T_2} \zeta'(\rho)^{-1}
   |A(\rho)|^2  \ \mathrm{and} \
   S_{2} = \sum_{T_1 < \gamma < T_2} |A(\rho)|^2 
\]
where $A(s) = \sum_{k \le M} x_{k} k^{-s}$
and $x_k$ is an arbitary real sequence.
Here we choose $T_1,T_2$ such that 
\[
     \zeta(\sigma+iT_j)^{-1} \ll T_{j}^{\epsilon}
\]
where $T_1 = T+O(1)$ and $T_2 = 2T + O(1)$.  
This is possible by 
Theorem 14.16 of \cite{Ti}. 
We shall establish:
\begin{proposition} \label{prop2}
Assume the Riemann hypothesis and that all but finitely many of the zeros of the Riemann zeta function
are simple.  If $||\frac{x_n}{n}||_1 \ll T^{\epsilon}$ 
\[
   S_3 =  \frac{T_2-T_1}{2 \pi} \sum_{hn \le M} \frac{\mu(n)x_{h}x_{nh}}{nh} 
  + O \left(  T^{\epsilon}  ( M||x_n||_{\infty} + 
  ||x_n||_1 + T^{\frac{1}{2}} ||x_{n}^2||_{1}^{\frac{1}{2}})
   \right)   
\]
for $T$ sufficiently large. 
\end{proposition}
Moreover by Proposition \ref{prop1} we have
\[
   S_2 = 
    (N(T_2)-N(T_1)) \sum_{m \le M} \frac{x_{m}^2}{m}
    - \frac{T_2-T_1}{\pi} \sum_{m \le M} \frac{(\Lambda*x)(m)x_{m}} {m}
    + o(T)  
\]
respectively.
We now choose $x_m = \sqrt{m} \mu(m) f(m)$ and suppose that 
$M < T^{2/3-10\epsilon}$.  Note that $||x_n||_{\infty} \ll 
M^{\frac{1}{2}+\epsilon}, ||x_n||_1 \ll M^{1+\epsilon}$ and thus
\[
   S_{3} = \frac{T_2-T_1}{ 2\pi}
   \left( \sum_{hn \le M} \frac{f(h) f(nh)}{\sqrt{n}} + 
   o(1) \right)   
\]
and 
\[
   S_{2} =  
   ( N(T_2)-N(T_1))
   \sum_{m \le M} f(m)^2 - \frac{T_2-T_1}{\pi} \sum_{m \le M} 
   \frac{(\Lambda*f)(m) f(m)}{m}
   + o(T) \ . 
\]
The second sum in $S_2$ is bounded by 
\[
   \sum_{mp \le M} \frac{(\log p) f(m) f(mp)}{\sqrt{p}}
  \ll \sum_{p \le b} \frac{\log p f(p)}{\sqrt{p}}
  \sum_{{\begin{substack}{m \le M/p
         \\ (m,p)=1}\end{substack}}}
   f(m)^2  \\
  \ll \left( \sum_{m \le M} f(m)^2 \right)  (\log T)^{1/2+\epsilon}   
  \ . 
\]   
With these observations in hand we obtain
\[
   \max_{T \le \gamma \le 2T}
   |\zeta'(\rho)|^{-1}    
     \ge \Lo^{-1} \left(\frac{  \sum_{hn \le M} \frac{f(h)f(nh)}{\sqrt{n}} 
     }
    {  \sum_{m \le M} f(m)^2}
    \right)
    (1+o(1))
\]
and by Soundararajan's calculation we obtain 
\[
   \max_{T \le \gamma \le 2T}
   |\zeta'(\rho)|^{-1}    \ge 
   \exp \left( (1+o(1))
   \mbox{$ \sqrt{\frac{\log M}{\log \log M}}$}
   \right) 
\]
for $M < T^{\frac{2}{3}-10\epsilon}$ which yields (\ref{eq:lb}).

\noindent It now suffices to establish Proposition \ref{prop2}.  \\

\noindent {\it Proof of Proposition \ref{prop2}.}
We consider the integral 
\[   I :=  \frac{1}{2 \pi i} 
   \int_{c+iT_1}^{c+iT_2} \zeta(s)^{-1} A(s)A(1-s) \, ds  \ . 
\]
with $c =1+ O((\log T)^{-1})$.  Moving the contour left to the $1-c$ line yields
$I=S_3 + H+I'+O(1)$ where
\[
    I' :=  \frac{1}{2 \pi i}
    \int_{1-c+iT_1}^{1-c+iT_2} \zeta(s)^{-1} A(s)A(1-s) \, ds  
\]
and $H$ are the horizontal contributions.   We know from Proposition \ref{prop1} 
that 
\[
    I = \frac{T_2-T_1}{2 \pi} \sum_{nu \le M} \frac{\mu(n)x_{u}x_{nu}}{nu}
   + O(M^{\epsilon}( ||x_n||_{\infty} M + ||x_n||_1)) \ . 
\]
Next we consider the contribution from the horizontal terms.  
We may verify that $|A(s)A(1-s)| \le M ||\frac{x_n}{n}||_{1}^2
+ ||x_n||_1 ||\frac{x_n}{n}||_1$ for $1-c \le \mathrm{Re}(s) \le c$. 
Furthermore, since we have chosen the $T_j$ such that $\zeta(\sigma +iT_{j})^{-1} \ll T_{j}^{\epsilon}$,  $H \ll   T^{\epsilon} (M + ||x_n||_1)$.
We now consider the contribution of the left hand side. 
We have that $\zeta(s)=\chi(s)\zeta(1-s)$.  Since $\chi(s) \asymp
T^{1/2}$ and $\zeta(1-s) \asymp \log T$ for $\mathrm{Re}(s)=1-c$
we have 
\begin{equation}
\begin{split}
  I' &  
  \ll (T^{1/2} \log T)^{-1} ||\mbox{$\frac{x_n}{n}$}||_1
  \int_{T_1}^{T_2} |A(1-c+it)| \, dt   \\
   &  \ll (\log T )^{-1} ||\mbox{$\frac{x_n}{n}$}||_1
    \left(\int_{T_1}^{T_2} |A(1-c+it)|^2 \, dt
  \right)^{1/2}  \ . 
  \nonumber
\end{split}
\end{equation}
The mean value theorem for Dirichlet polynomials asserts
\[
    \int_{T_1}^{T_2} \left| 
    \sum_{n \le N} \frac{a_n}{n^{it}}
    \right|^2 \, dt 
    =  (T_2-T_1) \sum_{n \le N} |a_n|^2 + 
    O  \mbox{$ \left( 
    \sum_{n \le N} n |a_n|^2
    \right) $} \ . 
\]
Since $1-c= O((\log T)^{-1})$
\[
    \int_{T_1}^{T_2} |A(1-c+it)|^2 \, dt 
    \ll  T 
    \sum_{n \le M} x_{n}^2
    + \sum_{n \le M} \frac{x_{n}^2}{n}
    \ll T  ||x_{n}^2||_1  \ . 
\]
Thus we deduce that $I' \ll  T^{\frac{1}{2}}(\log T)^{-1}
||\frac{x_n}{n}||_1 ||x_{n}^2||_{1}^{\frac{1}{2}}$. 
Collecting estimates yields Proposition \ref{prop2}. $\Box$


\noindent {\it Acknowledgements}.
The author thanks Professor Soundararajan for suggesting this problem.

\noindent Nathan Ng\\
Department of Mathematics and Statistics \\
University of Ottawa \\
585 King Edward Ave. \\
Ottawa, ON \\
Canada K1N6N5 \\
\email{nng@uottawa.ca}
%
\end{document}